%% file: R-Kingman.tex
\documentclass[10pt,a4paper]{article}
\usepackage{amsmath,amsfonts,amsthm,amssymb}
\usepackage{amsmath,amsfonts,amsthm}
\usepackage{amssymb}
\usepackage[top=3cm,left=3cm,right=2.5cm,bottom=2.5cm]{geometry}

\usepackage{graphicx,color}
\usepackage[usenames,dvipsnames,svgnames,table]{xcolor}
\usepackage[colorlinks = true, linkcolor = blue, urlcolor  = blue, citecolor = violet]{hyperref}

\numberwithin{equation}{section}

\newtheorem{corollary}{Corollary}[section]

\newtheorem{lemma}[corollary]{Lemma}
\newtheorem{theorem}[corollary]{Theorem}

\theoremstyle{definition}
\newtheorem{remark}[corollary]{Remark}

\newcommand{\numset}[1]{\mathbb{#1}}
	\newcommand{\nn}{\numset{N}}

	\newcommand{\one}{\boldsymbol{1}}
		\newcommand{\Exp}[1]{\mathrm{e}^{#1}}
	
	\makeatletter
	\providecommand*{\diff}%
	{\@ifnextchar^{\DIfF}{\DIfF^{}}}
	\def\DIfF^#1{%
	\mathop{\mathrm{\mathstrut d}}%
	\nolimits^{#1}\gobblespace}
	\def\gobblespace{%
	\futurelet\diffarg\opspace}
	\def\opspace{%
	\let\DiffSpace\!%
	\ifx\diffarg(%
	\let\DiffSpace\relax
	\else
	\ifx\diffarg%
	\let\DiffSpace\relax
	\else
	\ifx\diffarg\{%
	\let\DiffSpace\relax
	\fi\fi\fi\DiffSpace}

	\renewcommand{\d}{\diff}

\title%[A gapped subadditive ergodic theorem]
  {A gapped generalization of \\ Kingman's subadditive ergodic~theorem}

\author%[R.~Raqu\'{e}pas]
  {Renaud Raqu\'{e}pas}
\date{\today}

\begin{document}

\maketitle

{\small
\begin{center}
  New York University \\
  Courant Institute of Mathematical Sciences
  \\
  251 Mercer Street, New York \\
	NY 10012, United States \\
\end{center}
}

\begin{abstract}
  We state and prove a generalization of Kingman's ergodic theorem on a measure-preserving dynamical system~$(X,\mathcal{F},\mu,T)$ where the $\mu$-almost sure subadditivity condition
  \[
    f_{n+m} \leq f_n + f_m \circ T^{n}
  \]
  is relaxed to a $\mu$-almost sure, ``gapped'', almost subadditivity condition of the form
  \[
    f_{n+\sigma_m+m} \leq f_n +\rho_n + f_m \circ T^{n+\sigma_n}
  \]
  for some nonnegative~$\rho_n \in L^1(\d\mu)$ and $\sigma_n \in \nn \cup \{0\}$ that are suitably sublinear in~$n$. This generalization has a first application to the existence of specific relative entropies for suitably decoupled measures on one-sided shifts.\\


\end{abstract}

% \tableofcontents

\newcommand{\bX}{\mathbf{X}}
\newcommand{\bY}{\mathbf{Y}}

\section{Introduction}

In this short note, we consider a general measurable space~$(X, \mathcal{F})$, on which a probability measure~$\mu$ is invariant for some measurable transformation~$T : X \to X$, and state and prove an ergodic theorem for sequences~$(f_n)_{n\in\nn}$ of measurable functions on~$(X,\mathcal{F})$ satisfying what we will call a ``gapped almost subadditivity condition''.
It is a  generalization of Kingman's subadditive ergodic theorem~\cite{Ki} that\,---\,to the author's knowledge\,---\,has not been pointed out in the literature. The basic theorem, Theorem~\ref{thm:as-K}, yields $\mu$-almost sure convergence of $\tfrac 1n f_n$ as $n\to\infty$, but may be strengthened to convergence in~$L^1(\d\mu)$ under stronger assumptions using Theorem~\ref{thm:L1-K}.

For the author of the present note, the interest in such a generalization stemmed from the study of estimators of specific relative entropies for pairs of shift-invariant Borel probability measures on the space~$\mathcal{A}^\nn$ of sequences $x = (x_k)_{k\in\nn}$ with values in some finite alphabet~$\mathcal{A}$~\cite{CDEJR,CDEJRb}. The focus there is on measures satisfying decoupling-type conditions, already exploited in works such as~\cite{Pf} by Ch.-\'{E}.~Pfister as well as~\cite{CJPS19} by N.~Cuneo, V.~Jak\v{s}i\'{c}, C.-A.~Pillet and A.~Shirikyan on the theory of large deviations.
Such conditions become gapped subadditivity-type conditions upon taking logarithms. As a concrete example, consider the following upper-decoupling property for a probability measure~$\mathbb{Q}$ on~$\mathcal{A}^\nn$:
\begin{quote}
	\noindent There exist $o(n)$-sequences $(c_n)_{n\in\nn}$ and $(\tau_n)_{n\in\nn}$ such that
	\[
		\mathbb{Q}\{ x : x_1^n = a, x_{n+\tau_n+1}^{n+\tau_n+m} =  b\}
			\leq \Exp{c_n} \mathbb{Q}\{x : x_1^n = a\}\mathbb{Q}\{x : x_{n+\tau_n+1}^{n+\tau_n+m} =  b\}
	\]
	for all~$a \in \mathcal{A}^n$, $n\in\nn$, $b \in \mathcal{A}^m$ and $m\in\nn$.
\end{quote}
Indeed, the functions $f_n: x \mapsto \log \mathbb{Q}_n(x_1^n)$ then define a gapped, almost subadditive sequence in the sense of the present note. Hence, by corollary of Theorems~\ref{thm:as-K} and~\ref{thm:L1-K} (and the Shannon--McMillan--Breiman theorem), we have the following result. We refer to~\cite{CJPS19,CDEJR,CDEJRb} for thorough discussions of the range of applicability of the upper decoupling condition and its generalizations, including adaptations to countably infinite alphabets. \\

\noindent\textbf{Corollary} (Special case of~\cite[\S{5}]{CDEJR}).
	{\em
	Let $\mathbb{P}$ and~$\mathbb{Q}$ be shift-invariant probability measures on~$\mathcal{A}^\nn$. If $\mathbb{Q}$ satisfies the above upper-decoupling condition, then the \textnormal{(}possibly infinite\textnormal{)} $\mathbb{P}$-almost sure limit 
	\[
		h_{\mathbb{P},\mathbb{Q}}(x) := \lim_{n\to\infty} \frac 1n \log \mathbb{Q}_n\{y : y_1^n = x_1^n\}
	\]
	exists, and so does the \textnormal{(}possibly infinite\textnormal{)} specific relative entropy of~$\mathbb{P}$ with respect to~$\mathbb{Q}$.
	}

\section{Gapped subadditivity and Fekete's lemma}

The main difference with previous works that the author is aware of is that the subadditivity-type condition under consideration for the ergodic theorem below allows for what we will call ``gaps''. Both as an illustration and as a technical building block, we note the following generalization of Fekete's lemma.

\begin{lemma}
\label{lem:gapped-F}
	Let $(F_n)_{n\in\nn}$ be a sequence\footnote{We choose the convention that~$\nn = \{1,2,3,\dotsc\}$.} in~$[-\infty,\infty)$ and suppose that
  there exist $o(n)$-sequences $(\sigma_n)_{n\in\nn}$ and $(R_n)_{n\in\nn}$ of nonnegative integers such that
	\begin{equation}
  \label{eq:gapped-for-Fekete}
    F_{n+\sigma_n+m} \leq F_n + R_n + F_m,
  \end{equation}
	for all $n,m \in \nn$. Then, with
  \begin{equation}
  \label{eq:liminf-Fekete}
    F := \inf_{n\in\nn}\frac{F_n + R_n}{n + \sigma_n}
  \end{equation}
  understood in~$[-\infty,\infty)$, we have
	\[
		\lim_{n\to\infty} \frac{F_n}{n} = F.
	\]
\end{lemma}

\begin{proof}
  Since the assumptions on~$(R_n)_{n\in\nn}$ and $(\sigma_n)_{n\in\nn}$ and the definition of~$F$ imply that
  \[
    \liminf_{n\to\infty} \frac{F_n}{n}
    = \liminf_{n\to\infty} \frac{F_n + R_n}{n+\sigma_n}
    \geq \inf_{n\in\nn} \frac{F_n + R_n}{n+\sigma_n} = F,
  \]
  it suffices to show that
  \begin{equation}
  \label{eq:to-show-Fekete}
      \limsup_{n\to\infty} \frac{F_n}{n} \leq F.
  \end{equation}
  To this end, let $F' > F$ be arbitrary. By definition of the infimum under consideration, there exists~$r\in\nn$ such that
  \[
    {F_r + R_r} \leq F'({r + \sigma_r}).
  \]
  Now, note that an arbitrary number $n \in \nn$, may be written as
  $
    n = k_n(r+\sigma_r) + q_n
  $
  for some $k_n \in \nn$ and $q_n \in \{1, 2, \dotsc, r + \sigma_r\}$. Hence, using~\eqref{eq:gapped-for-Fekete} repeatedly,
  \begin{align*}
    F_n &\leq k_n(F_r + R_r) + F_{q_n}  \\
      &\leq k_n F' (r+\sigma_r) + \max\{F_{q,+} : q \in \{1, 2, \dotsc, r + \sigma_r\} \}.
  \end{align*}
  Therefore,
  \begin{align*}
    \frac{F_n}{n}
      &\leq \frac{k_n(r+\sigma_r)}{k_n(r+\sigma_r) + q_n} F'  + \frac{\max\{F_{q,+} : q \in \{1, 2, \dotsc, r + \sigma_r\} \}}{n}.
  \end{align*}
  Because $k_n \to \infty$ as $n\to\infty$ and because $q_n \leq r + \sigma_r$ independently of~$n$, we may deduce that
  \begin{align*}
    \limsup_{n\to\infty}\frac{F_n}{n}
      &\leq F'.
  \end{align*}
  Since $F' > F$ was arbitrary, we conclude that the inequality~\eqref{eq:to-show-Fekete} indeed holds.
\end{proof}

In what follows, the terms ``gap'' and ``gapped'' will refer to the possibly nonzero integers~$\sigma_m$ that appear in conditions such as~\eqref{eq:gapped-for-Fekete}\,---\,this has \emph{nothing} to do with the spectral-theoretic notion of ``gap''. In this terminology, the usual statement of Fekete's lemma is the gapless case.

\section{Main result}

Several decades after M.~Fekete's lemma, J.F.C.~Kingman's ergodic theorem highlighted the role of (gapless) subadditivity in ergodic theory and dynamical systems~\cite{Ki}. Starting in the 1980s, new proofs of the original result paved the way for several key (gapless) generalizations, most notably by Y.~Derriennic~\cite{De} and by K.~Sch\"urger~\cite{Sc}.
Our main result is a gapped version of such a generalization of J.F.C.~Kingman's theorem.

\begin{theorem}
\label{thm:as-K}
	Let $(f_n)_{n\in\nn}$ be a sequence of measurable functions with positive parts~$f_{n,+} \in L^1(\d\mu)$ for all~$n\in\nn$.
	Suppose in addition that there exists an $o(n)$-sequence $(\sigma_n)_{n\in\nn}$ of nonnegative integers with $\sigma_1 = 0$ and a sequence $(\rho_n)_{n\in\nn}$ of nonnegative functions in~$L^1(\d\mu)$ such that
	\begin{equation}
	\label{eq:isu}
	\begin{split}
		f_{n+\sigma_n+m} &\leq f_n + \rho_n + f_m \circ T^{n+\sigma_n}
	\end{split}
	\end{equation}
	$\mu$-almost surely for all $n,m \in \nn$, and such that
	\begin{equation}
	\label{eq:as-rho-cond}
		\lim_{n\to\infty}\frac{\rho_n}{n} = 0
	\end{equation}
	$\mu$-almost surely. Then, the limit
	\begin{equation}
	\label{eq:as-K}
		f = \lim_{n\to\infty} \frac{f_n}{n}
	\end{equation}
	exists $\mu$-almost surely.
\end{theorem}

Before we proceed with the proof of this almost sure convergence\,---\,in steps mostly inspired by K.~Sch\"urger's sequence of arguments in~\cite[\S{2}]{Sc}, but also to some extent by J.M.~Steele~\cite[\S{2}]{St} and by A.~Avila and J.~Bochi~\cite{AB}\,---, we briefly comment on its hypotheses:
  it is condition~\eqref{eq:isu} that we call a gapped almost subadditivity condition,
  and
  the condition~$\sigma_n = o(n)$ controls the size of the gaps in this gapped condition.
Let us now proceed.

\begin{proof}
	We set
	\[
		f := \liminf_{n\to\infty} \frac{f_n}{n}
	\]
	and want to show that, $\mu$-almost surely, this limit inferior coincides with the corresponding limit superior.
	To do so, we introduce an arbitrarily small parameter $\epsilon > 0$.

	\begin{description}
    \item[Step 1] Let us first show that, as a consequence of~\eqref{eq:isu} with $n=1$ and $\sigma_1 = 0$, we have the identity
    \begin{equation}
    \label{eq:f-T-inv}
      f \circ T = f
    \end{equation}
		in the almost sure sense.
    First, note that taking~$m\to\infty$ there, we find
    \begin{align*}
      f &\leq \liminf_{m\to\infty} \left( \frac{f_{1,+}}{m} + \frac{\rho_1}{m} + \frac{f_{m} \circ T}{m}\right) \\
        &= f \circ T
    \end{align*}
    in the almost sure sense.
    But then, we have the inclusion $\{x : f(Tx) \geq y\} \supseteq \{x : f(x) \geq y\}$ for all~$y \in [-\infty,\infty)$, and since
    \[
      \mu\{x : f(Tx) \geq y\} = (\mu \circ T^{-1})\{x : f(x) \geq y\} = \mu\{x : f(x) \geq y\}
    \]
    by $T$-invariance of the measure~$\mu$,
    we conclude that
    \[
      \mu(\{x : f(Tx) \geq y\} \triangle \{x : f(x) \geq y\}) = 0
    \]
    for all~$y \in [-\infty,\infty)$, and therefore that~\eqref{eq:f-T-inv} indeed holds.

		\item[Step 2] Let us show that, for $r\in\nn$ large enough, almost surely, there exists~$k$ such that
		\begin{equation}
			\frac{f_{kr} + \rho_{kr}}{kr+\sigma_{kr}} \leq \max\left\{f,-\epsilon^{-1}\right\} + \epsilon.
		\end{equation}
		Given any $n \in \nn$, there exists a natural number~$k_n$ such that
		\begin{equation}
		\label{eq:choice-of-nk}
			(k_n-1)r \leq n + \sigma_n < k_n r,
		\end{equation}
		and then, by~\eqref{eq:isu}, we have the almost sure inequality
		\begin{align}
    \label{eq:f-knr}
			f_{k_n r}
				&\leq f_{n} + \rho_n + f_{k_n r - n - \sigma_n} \circ T^{n+\sigma_n}
		\end{align}
    Note that
    \begin{align*}
      f_{k_n r - n - \sigma_n} \circ T^{n+\sigma_n} \leq \sum_{q=1}^r f_{q,+} \circ  T ^{n + \sigma_n} .
    \end{align*}
	by~\eqref{eq:choice-of-nk}.
		Hence, dividing~\eqref{eq:f-knr} by~$n$ and taking $n \to\infty$,
		\begin{align*}
			\liminf_{n\to\infty}\frac{f_{k_n r}}{n}
			&\leq
			\liminf_{n\to\infty}\left( \frac{f_n}{n} + \frac{\rho_n}{n}
				+ \sum_{q=1}^r \frac{f_{q,+} \circ  T ^{n + \sigma_n}}{n}
				\right),
		\end{align*}
		almost surely. In view of~\eqref{eq:as-rho-cond},~\eqref{eq:choice-of-nk}, and a standard consequence of Birkhoff's theorem\footnote{See \emph{e.g.}~Lemma~2 in~\cite{AB} for a direct proof. This lemma is to be applied to each of the finitely many integrable functions of the form~$f_{q,+}$ using the fact that $\sigma_n = o(n)$.}, we have the almost sure inequality
		\begin{align*}
			\liminf_{n\to\infty}\frac{f_{k_n r}}{k_n r}
			&\leq
			\liminf_{n\to\infty} \frac{f_n}{n}.
		\end{align*}
		Recall that the right-hand side is the definition of~$f$. Hence, for $r$ large enough, there almost surely exist infinitely many~$k$ such that
		\begin{equation*}
			\frac{f_{kr}}{kr+\sigma_{kr}} \leq \max\left\{f,-\epsilon^{-1}\right\} + \frac{\epsilon}{2}.
		\end{equation*}
		{The criterion on~$r$ can be determined as a function of $\epsilon$ and $(\sigma_r)_{r\in\nn}$ only.}
    	But, again almost surely,
		\begin{equation*}
			\frac{\rho_{kr}}{kr+\sigma_{kr}}  \leq \frac{\epsilon}{2}
		\end{equation*}
		for all~$k$ large enough in view of~\eqref{eq:as-rho-cond} and nonnegativity of~$\rho_{kr}$. This gives the desired conclusion.

		\item[Step 3] Let us show that, for $r \in \nn$ large enough,
		the sets
		\[
			D^{r,K,\epsilon} :=
			\bigcap_{k=1}^{K}
			\left\{  x :
				\frac{f_{kr}( x ) + \rho_{kr}( x )}{kr+\sigma_{kr}} > \max\left\{f( x ),-\epsilon^{-1}\right\} + \epsilon
			\right\}
		\]
		are such that
		\[
			\lim_{K\to\infty}\lim_{n\to\infty} \frac{1}{n}\sum_{j=0}^{n-1} (\one_{D^{r,K,\epsilon}}(1+f_{r,+} + \rho_{r})) \circ T^j = 0
		\]
		almost surely. To do so, we omit some indices and let
		\[
			\psi_{n,K} := \frac{1}{n}\sum_{j=0}^{n-1} (\one_{D^{K}}(1+f_{r,+} + \rho_{r})) \circ T^j.
		\]
		By Birkhoff's theorem, the limit
		\[
			\psi_{K} =
			\lim_{n\to\infty} \psi_{n,K}
		\]
		almost surely exists and
		\[
			\int \psi_{K} \d\mu = \int_{D^{K}} (1+f_{r,+} + \rho_{r}) \d\mu.
		\]
		On the other hand, Lebesgue dominated convergence guarantees that
		\[
			\int \lim_{K\to\infty} \psi_{K} \d\mu
			=
			\lim_{K\to\infty} \int \psi_{K} \d\mu.
		\]
		Therefore, because $\psi_{K}$ is nonnegative, we need only show that the right-hand side vanishes.
		But, using once again Lebesgue dominated convergence,
		\[
			\lim_{K\to\infty} \int \psi_{K} \d\mu = \lim_{K\to\infty} \int_{D^{K}} (1+f_{r,+} + \rho_{r}) \d\mu = \int_{\bigcap_{K \in \nn} D^{K}} (1+f_{r,+} + \rho_{r}) \d\mu,
		\]
		so this follows from Step~2 provided that $r$ is large enough.

		\item[Step 4] Let us show that
		\begin{align*}
			\limsup_{n\to\infty} \frac{f_n}{n}
				&\leq \max\left\{f,-\epsilon^{-1}\right\} + \epsilon.
		\end{align*}
		almost surely. To do so, fix $r$ large enough and~$x$ in a set of measure~$1$ where conclusions of Steps~1 and~3 hold, and let us construct a gapped Steele-type collection~$(I_\ell)_{\ell\in\nn}$ of ordered disjoints intervals in~$\nn$ based on the behaviour of the trajectory starting at~$x$.

		\begin{description}
			\item[Base case] Set $I_0 = \{0\}$.
			\item[Induction] Suppose that we have constructed~$I_0$ up to $I_\ell$ and let $m_\ell$ be the maximum of~$I_\ell$.
      \emph{Case~1.} If $T^{m_\ell} x  \notin D^{r,K,\epsilon}$,  let $I_{\ell+1} = [m_\ell  + 1, m_\ell  + k_{\ell+1}r + \sigma_{k_{\ell+1} r}] \cap \nn$, where $k_{\ell+1}$ is the smallest~$k \in \{1,2, \dotsc K\}$ such that
			\begin{equation}
      \label{eq:f-bound-case-1}
        \frac{f_{kr}(T^{m_\ell} x ) + \rho_{kr}(T^{m_\ell} x )}{kr+\sigma_{kr}} \leq \max\left\{f(T^{m_\ell} x ),-\epsilon^{-1}\right\} + \epsilon.
      \end{equation}
			\emph{Case~2.} If $T^{m_\ell} x  \in D^{r,K,\epsilon}$, let $I_{\ell+1} = [m_\ell  + 1, m_\ell + r + \sigma_{r}] \cap \nn$.
		\end{description}

    Given~$n \in \nn$ large enough, we may split relevant indices into two categories\footnote{Informally, $\ell \in \mathsf{G}_n$ is ``good'' because the term $f_{k_\ell r}$ appearing in the estimate below can be bounded in terms of~$f$ and~$\epsilon$ directly; $\ell \in \mathsf{B}_n$ is ``bad'' because the corresponding term $f_{r,+} $ cannot\,---\,but this is rare enough for $n \gg K \gg 1$ in view of Step~3.}
    \[
			\mathsf{G}_n = \left\{ \ell : I_\ell \subseteq [1,n - 1] \text{ and } I_\ell \text{ follows Case~1}\right\}.
		\]
    and
    \[
			\mathsf{B}_n = \left\{ \ell : I_\ell \subseteq [1,n - 1] \text{ and } I_\ell \text{ follows Case~2}\right\}.
		\]
		Now, we set $M_n := m_{\max\{\mathsf{B}_n \cup \mathsf{G}_n\}}$ and use~\eqref{eq:isu} repeatedly along these intervals to write
		\begin{equation}
    \label{eq:ub-rep}
    \begin{split}
      f_n( x )
      \leq  \sum_{\ell \in \mathsf{G}_n} (f_{k_\ell r} + \rho_{k_\ell r})(T^{m_{\ell-1}}x)
      + \sum_{\ell \in \mathsf{B}_n} (f_{r}
      + \rho_r)(T^{m_{\ell-1}} x )
      % \qquad
      % \\
      % & \qquad\qquad{}
      + f_{n-M_n,+} (T^{M_n}x);
    \end{split}
		\end{equation}
    see Figure~\ref{fig:gapped-steele-cut}.
    \begin{figure}
      \center
      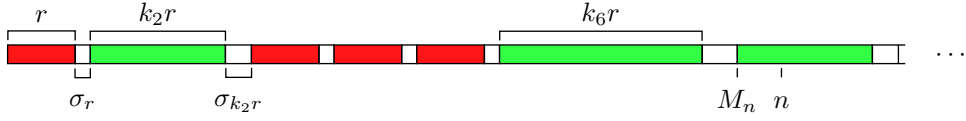
      \caption{Illustration of the gapped Steele-type collection of intervals in Step~4 of the proof of Theorem~\ref{thm:as-K}: green corresponds to Case~1 (deemed ``good'') and red corresponds to Case~2 (deemed ``bad''). The construction of these blocks along which we exploit the gapped almost subadditivity condition depends on~$x$, $r$, $K$ and $\epsilon$.}
      \label{fig:gapped-steele-cut}
    \end{figure}
    Note that
    \begin{align*}
      \sum_{\ell \in \mathsf{G}_n} (f_{k_\ell r} + \rho_{k_\ell r})(T^{m_{\ell-1}}x)
      \leq \left(\max\left\{f(x),-\epsilon^{-1}\right\} + \epsilon\right) \sum_{\ell \in \mathsf{G}_n} (k_\ell r+\sigma_{k_\ell r})
    \end{align*}
    by~\eqref{eq:f-bound-case-1} and~\eqref{eq:f-T-inv},
    \begin{align*}
      \sum_{\ell \in \mathsf{B}_n} (f_{r}
      + \rho_r)(T^{m_{\ell-1}} x ) \leq \sum_{j = 0}^n (\one_{D^{r,K,\epsilon}} (f_{r,+}
      + \rho_r))(T^{j} x )
    \end{align*}
    by nonnegativity,
    and
    \begin{align*}
      f_{n-M_n,+} (T^{M_n}x) \leq \sum_{q=1}^{Kr + \bar{\sigma}_{Kr}}f_{q,+} (T^{M_n}x),
    \end{align*}
    where $\bar{\sigma}_{Kr} := \max \{\sigma_{kr} : k = 1, 2, \dotsc, K\}$ is independent of~$n$.
		Therefore, appealing to Step~3 and Birkhoff's theorem as used in Step~2, we have the bound
		\[
			\limsup_{n\to\infty} \frac{f_n( x )}{n}
				\leq \limsup_{K\to\infty}\limsup_{n\to\infty}
				\frac{\left(\max\left\{f( x ),-\epsilon^{-1}\right\} + \epsilon\right) \sum_{\ell \in \mathsf{G}_n} (k_\ell r+\sigma_{k_\ell r})}{n}.
		\]
		Note that, by construction,
		\begin{align*}
			n &\leq \sum_{\ell \in \mathsf{G}_n} (k_\ell r + \sigma_{k_\ell r}) + \sum_{\ell \in \mathsf{B}_n} (r+\sigma_r)
      + (n - M_n)
      \\
      &\leq \sum_{\ell \in \mathsf{G}_n} (k_\ell r + \sigma_{k_\ell r}) + \sum_{j=0}^{n} (r+\sigma_r) \one_{D^{r,K,\epsilon}} \circ T^j
      + Kr + \bar{\sigma}_{Kr},
		\end{align*}
    so that
		\begin{align*}
			 n - (r+\sigma_r) \sum_{j=0}^{n} \one_{D^{r,K,\epsilon}} \circ T^j - Kr - \bar{\sigma}_{Kr} \leq \sum_{\ell \in \mathsf{G}_n} (k_\ell r + \sigma_{k_\ell r}) \leq n.
		\end{align*}
		Therefore, % for $\epsilon > 0$ small enough,
		\begin{align*}
			\limsup_{n\to\infty} \frac{f_n( x )}{n}
				&\leq \max\left\{f( x ),-\epsilon^{-1}\right\} + \epsilon
					\\
					&\qquad\quad{} + \epsilon^{-1} \limsup_{K\to\infty}\limsup_{n\to\infty} \left(\frac{Kr + \bar{\sigma}_{Kr}}{n} + \frac {r+\sigma_r}{n} \sum_{j=0}^{n} \one_{D^{r,K,\epsilon}} \circ T^j( x ) \right).
		\end{align*}
		Taking first $n\to\infty$, and then $K\to\infty$ using again the conclusion of Step~3, we obtain the desired inequality.
	\end{description}
	The fact that the conclusion of Step~4 holds $\mu$-almost surely for an arbitrarily small $\epsilon > 0$, together with the fact that~$f$ was initially defined as the corresponding limit inferior, gives the existence of the desired limit, $\mu$-almost surely.
\end{proof}

\begin{remark}
  Note that the requirement that $\sigma_1 = 0$ was only used in deriving the $\mu$-almost sure inequality $f \leq f \circ T$ and can therefore be discarded if one knows \emph{a priori} that this inequality holds $\mu$-almost surely, \emph{e.g.}\ because $f_{n+1} \leq f_n \circ T$ for all~$n \in \nn$ large enough, as is the case in~\cite{CDEJR,CDEJRb} and in the corollary presented in the Introduction, where $f_n$ is the logarithm of the $n$-th marginal of a probability measure.
  In fact, there $f_{n+1} \leq f_n$ as well and this simplifies Steps~2 and~4 of the proof.
\end{remark}

An additional assumption on the sequence~$(\rho_n)_{n\in\nn}$ of error terms then allows to prove convergence at the level of the integrals. As announced in the introduction, the theorem we are about to state and prove gives convergence in~$L^1(\d\mu)$ using the Riesz--Scheff\'{e} lemma if the functions~$f_n$ all have a definite sign and if~$f \in L^1(\d\mu)$\,---\,the latter being equivalent in this context to the requirement that~$\inf_{n\in\nn} \tfrac 1n \int f_n\d\mu > -\infty$.

\begin{theorem}
\label{thm:L1-K}
	If, in addition to the hypotheses of Theorem~\ref{thm:as-K}, the sequence $(\rho_m)_{m\in\nn}$ satisfies
	\begin{equation}
	\label{eq:L1-rho-cond}
		\lim_{n \to\infty} \frac 1n \int \rho_n \d\mu = 0,
	\end{equation}
	then
	\begin{equation}
	\label{eq:avg-K}
		\lim_{n \to\infty} \frac 1n \int f_n \d\mu = \int f \d\mu,
	\end{equation}
	understood in~$[-\infty,\infty)$.
\end{theorem}

\begin{proof}
	We prove the two inequalities behind the desired equality~\eqref{eq:avg-K} separately using a temporary cutoff depending on an arbitrarily small parameter~$\epsilon > 0$.
	\begin{description}
		\item[Step 1]
		Integrating~\eqref{eq:isu} and exploiting invariance of the measure, we define a sequence of integrals\footnote{Since we have assumed in Theorem~\ref{thm:as-K} that $f_{n,+}$ is integrable, the integral of $f_n$ is well defined in~$[-\infty, \infty)$, which suffices for Lemma~\ref{lem:gapped-F}} by
    $
      F_n := \int f_n \d\mu
    $
    and apply Lemma~\ref{lem:gapped-F}
    with the error terms defined by
    $
      R_n := \int \rho_n \d\mu
    $
    being $o(n)$ thanks to~\eqref{eq:L1-rho-cond}.
		This observation yields
		\begin{align*}
			\lim_{n\to\infty} \frac 1n \int f_n \d\mu
			&= \inf_{n\in\nn} \frac 1{n+\sigma_n} \left(\int f_n \d\mu + \int \rho_n \d\mu \right) \\
			&= \inf_{n\in\nn} \inf_{\epsilon > 0} \frac 1{n+\sigma_n} \left(\int \max\left\{f_n, -n\epsilon^{-1}\right\} \d\mu + \int \rho_n \d\mu \right) \\
			&= \inf_{\epsilon > 0} \inf_{n\in\nn} \frac 1{n+\sigma_n} \left(\int \max\left\{f_n, -n\epsilon^{-1}\right\} \d\mu + \int \rho_n \d\mu \right)
		\end{align*}
		in view of Lebesgue monotone convergence in~$\epsilon$ at fixed~$n$ and integrability of the positive part of~$f_n$. Now, using again Lemma~\ref{lem:gapped-F} to deal with the inner infimum on the right-hand side,
		\begin{align}
		\label{eq:epsilon-vs-n}
			\lim_{n\to\infty} \frac 1n \int f_n \d\mu
			&= \lim_{\epsilon \to 0} \lim_{n\to\infty} \frac 1{n} \int \max\left\{f_n, -n\epsilon^{-1}\right\} \d\mu.
		\end{align}
  \item[Step 2]
		At fixed~$\epsilon > 0$, Fatou's lemma applies with $-\epsilon^{-1}$ as an integrable lower bound, yielding
		\begin{align}
		\label{eq:K-S1}
			 \liminf_{n\to\infty} \frac 1{n} \int \max\left\{f_n, -n\epsilon^{-1}\right\} \d\mu
			&\geq  \int \max\left\{f, -\epsilon^{-1}\right\} \d\mu.
		\end{align}
    Taking the positive part of both sides of~\eqref{eq:isu} and using the obvious bound for the positive part of a sum, we see that the sequence $(f_{n,+})_{n\in\nn}$ of positive parts also satisfies~\eqref{eq:isu}. Therefore, the same argument and the fact that the limit inferior on the left-hand side can be written as an infimum guarantee that $f_+ \in L^1(\d\mu)$.

    \item[Step 3]
      To obtain the opposite bound, we argue separately for the negative and positive parts. For the negative parts, the Lebesgue dominated convergence theorem applies at fixed~$\epsilon > 0$ and yields
      \begin{align}
  		\label{eq:K-S2-neg}
  		 \lim_{n\to\infty} \frac 1{n} \int \max\left\{-f_{n,-}, -n\epsilon^{-1}\right\} \d\mu
  			&= \int \max\left\{-f_-, -\epsilon^{-1}\right\} \d\mu.
  		\end{align}
      By our previous comment on the positive parts, we may take the analogue of~\eqref{eq:ub-rep} for positive parts, integrate and divide by~$n$ in order to derive the bound
      \begin{align*}
        \frac 1n \int f_{n,+} \d\mu
        &\leq \int f_{+} \d\mu + \epsilon
        + \int_{D_{+}^{r,K,\epsilon}} (f_{r,+}
        + \rho_r) \d\mu
        + \frac{1}{n} \sum_{q = 1}^{Kr + \bar{\sigma}_{Kr}} \int f_{q,+} \d\mu.
      \end{align*}
      Taking $n \to \infty$,
      \begin{align*}
          \limsup_{n\to\infty} \frac 1n \int f_{n,+} \d\mu
          &\leq \int f_{+} \d\mu + \epsilon
          + \int_{D_{+}^{r,K,\epsilon}} (f_{r,+}
          + \rho_r) \d\mu.
      \end{align*}
      Taking $K\to\infty$ using Step~3 of the proof of the previous theorem applied to the positive parts, we find
      \begin{align}
      \label{eq:K-S2-pos}
          \limsup_{n\to\infty} \frac 1n \int f_{n,+} \d\mu
          &\leq \int f_{+} \d\mu + \epsilon.
      \end{align}
      Combining~\eqref{eq:K-S2-neg} and~\eqref{eq:K-S2-pos}, we find
      \begin{align}
      \label{eq:K-S2}
          \limsup_{n\to\infty} \frac 1n \int \max\left\{f_{n}, -n\epsilon^{-1}\right\} \d\mu
          &\leq \int \max\left\{f, -\epsilon^{-1}\right\} \d\mu + \epsilon.
      \end{align}
	\end{description}

  Combining~\eqref{eq:epsilon-vs-n} with~\eqref{eq:K-S1} and~\eqref{eq:K-S2}, we find
  \[
    \lim_{n\to\infty} \frac 1n \int f_n = \lim_{\epsilon \to 0} \int \max\left\{f, -\epsilon^{-1}\right\} \d\mu.
  \]
  Because $f_+$ has already been shown to be integrable, we now only need the Lebesgue monotone convergence theorem to conclude that
	\[
		 \lim_{\epsilon\to 0} \int \max\left\{f, -\epsilon^{-1}\right\} \d\mu = \int f \d\mu,
	\]
	and thus that the theorem holds.
\end{proof}

\begin{remark}
	While it is certainly natural to seek generalizations of the main results of this note along the lines of the passage from~(DS) to~(AS) in~\cite{Sc}, one must be aware that the gaps that are allowed here complicate the construction of the Steel-type intervals for a fixed~$x$ in Step~4 of the proof of Theorem~\ref{thm:as-K}. For example, it is already unclear to the author how to adapt the construction if one replaces~\eqref{eq:isu} with $f_{n+\sigma_m+m} \leq f_n + (f_m+\rho_m) \circ T^{n+\sigma_m}$. In the theory of large deviations for measures on one-sided shifts mentioned in the Introduction, the role of such ``exchanges of~$n$ and $m$'' is still not completely understood.
\end{remark}

\noindent
\paragraph*{Acknowledgements.} The author wishes to thank N.~Barnfield, G.~Cristadoro, N.~Cuneo, M.~Degli Esposti, V.~Jak\v{s}i\'{c} and J.~Soliman for stimulating discussions on the topic of this note. The author acknowledges financial support from the \emph{Natural Sciences and Engineering Research Council of Canada} and from the \emph{Fonds de recherche du Qu\'ebec\,---\,Nature et technologies}. A significant part of this work was done while the author was a post-doctoral researcher at CY Cergy Paris Universit\'e and supported by the LabEx MME-DII (\emph{Investissements d'Avenir} program of the French government).

\end{document}

%% file: fig-gapped-steele-cut.pdf_tex
%% Accompanies image file 'fig-gapped-steele-cut.pdf' (pdf, eps, ps)
\begingroup%
  \makeatletter%
  \providecommand\color[2][]{%
    \errmessage{(Inkscape) Color is used for the text in Inkscape, but the package 'color.sty' is not loaded}%
    \renewcommand\color[2][]{}%
  }%
  \providecommand\transparent[1]{%
    \errmessage{(Inkscape) Transparency is used (non-zero) for the text in Inkscape, but the package 'transparent.sty' is not loaded}%
    \renewcommand\transparent[1]{}%
  }%
  \providecommand\rotatebox[2]{#2}%
  \newcommand*\fsize{\dimexpr\f@size pt\relax}%
  \newcommand*\lineheight[1]{\fontsize{\fsize}{#1\fsize}\selectfont}%
  \ifx\svgwidth\undefined%
    \setlength{\unitlength}{373.7641338bp}%
    \ifx\svgscale\undefined%
      \relax%
    \else%
      \setlength{\unitlength}{\unitlength * \real{\svgscale}}%
    \fi%
  \else%
    \setlength{\unitlength}{\svgwidth}%
  \fi%
  \global\let\svgwidth\undefined%
  \global\let\svgscale\undefined%
  \makeatother%
  \begin{picture}(1,0.11719379)%
    \lineheight{1}%
    \setlength\tabcolsep{0pt}%
    \put(0,0){\includegraphics[width=\unitlength,page=1]{fig-gapped-steele-cut.pdf}}%
    \put(0.1504774,0.09215997){\color[rgb]{0,0,0}\makebox(0,0)[t]{\lineheight{0}\smash{\begin{tabular}[t]{c}$k_2r$\end{tabular}}}}%
    \put(0,0){\includegraphics[width=\unitlength,page=2]{fig-gapped-steele-cut.pdf}}%
    \put(0.23116094,0.00731859){\color[rgb]{0,0,0}\makebox(0,0)[t]{\lineheight{0}\smash{\begin{tabular}[t]{c}$\sigma_{k_2r}$\end{tabular}}}}%
    \put(0.77279217,0.00493609){\color[rgb]{0,0,0}\makebox(0,0)[t]{\lineheight{0}\smash{\begin{tabular}[t]{c}$n$\end{tabular}}}}%
    \put(0,0){\includegraphics[width=\unitlength,page=3]{fig-gapped-steele-cut.pdf}}%
    \put(0.03440495,0.09146403){\color[rgb]{0,0,0}\makebox(0,0)[t]{\lineheight{0}\smash{\begin{tabular}[t]{c}$r$\end{tabular}}}}%
    \put(0,0){\includegraphics[width=\unitlength,page=4]{fig-gapped-steele-cut.pdf}}%
    \put(0.07546549,0.00731859){\color[rgb]{0,0,0}\makebox(0,0)[t]{\lineheight{0}\smash{\begin{tabular}[t]{c}$\sigma_r$\end{tabular}}}}%
    \put(0,0){\includegraphics[width=\unitlength,page=5]{fig-gapped-steele-cut.pdf}}%
    \put(0.59257934,0.09215997){\color[rgb]{0,0,0}\makebox(0,0)[t]{\lineheight{0}\smash{\begin{tabular}[t]{c}$k_6r$\end{tabular}}}}%
    \put(0.72870923,0.00493634){\color[rgb]{0,0,0}\makebox(0,0)[t]{\lineheight{0}\smash{\begin{tabular}[t]{c}$M_n$\end{tabular}}}}%
    \put(0,0){\includegraphics[width=\unitlength,page=6]{fig-gapped-steele-cut.pdf}}%
    \put(0.94406485,0.05178666){\color[rgb]{0,0,0}\makebox(0,0)[t]{\lineheight{0}\smash{\begin{tabular}[t]{c}$\dotsb$\end{tabular}}}}%
    \put(0,0){\includegraphics[width=\unitlength,page=7]{fig-gapped-steele-cut.pdf}}%
  \end{picture}%
\endgroup%